\documentclass[11pt]{article}
\usepackage[english]{babel}
\usepackage[utf8]{inputenc}
\usepackage{fancyhdr}
\usepackage{amssymb,amsfonts,amsthm,amscd}
\usepackage{amsmath}
\usepackage[usenames]{color}
\usepackage{verbatim}

\usepackage{graphicx}
\usepackage{psfrag}

\newcommand{\bm}[1]{\mbox{\boldmath $#1$}}

\DeclareFontFamily{OT1}{rsfs}{}
\DeclareFontShape{OT1}{rsfs}{m}{n}{ <-7> rsfs5 <7-10> rsfs7 <10->
rsfs10}{} \DeclareMathAlphabet{\mycal}{OT1}{rsfs}{m}{n}

\newcommand{\mnotex}[1]
{\protect{\stepcounter{mnotecount}}$^{\mbox{\footnotesize $\bullet$\themnotecount}}$ 
\marginpar{
\raggedright\tiny\em
$\!\!\!\!\!\!\,\bullet $\themnotecount: #1} }

\newtheorem{Tma}{Theorem}
\newtheorem{Prop}{Proposition}
\newtheorem{Def}{Definition}
\newtheorem{Lem}{Lemma}
\newtheorem{Cor}{Corollary}
\newtheorem{remark}{Remark}
\theoremstyle{remark} \newtheorem*{Obs}{Remark}

\def\Journal#1#2#3#4#5#6{#1, ``#2'', {\em #3} {\bf #4}, #5 (#6).}

\def\JGP{\em J. Geom. Phys.}

\def\CQG{\em Class. Quantum Grav.}

\def\GRG{\em Gen. Rel. Grav.}

\def\CMP{\em Commun. Math. Phys.}

\def\PAMS{\em Proc. Amer. Math. Soc.}
\def\TAMS{\em Trans. Amer. Math. Soc.}

\def\AM{\em Ann. Math.}

\def\PJM{\em Pacific J. Math.}

\def\AHP{\em Annals Henri Poincar\'e}

\def\MM{\em Manuscripta Mathematica}
\def\GD{\em Geom. Dedicata}
\def\MJM{\em Mediterr. J. Math.}
\def\KJM{\em Kragujevac Journal of Mathematics}
\def\HPA{\em Helvetica Physica Acta} 
\def\JMPA{\em J. Math. pures et appl.}
\def\PRSE{\em Proc. Roy. Soc. Edinburgh A.}
\def\PSPM{\em Proc. Symp. Pure Math.}
\def\IM{\em Invent. Math.}
\def\RACCIEX{\em Revista de la Real Academia de Ciencias Exactas, Físicas y Naturales. Serie A. Matemáticas}
\def\DGA{\em Differential Geom. Appl.}
\def\SJM{\em Soochow J. Math.}

\def\gb{\overline{g}}
\def\gone{\tilde{g}}
\def\conb{\overline{\nabla}}
\def\conone{\tilde{\nabla}}
\def\norone{\tilde{N}}
\def\gtwo{\hat{g}}

\def\Mb{\overline{M}}
\def\gbase{g_B}

\def\conbase{\nabla^B}
\def\confiber{\nabla^F}
\def\cons{\nabla}

\def\vfield{\mathfrak{X}}

\def\vconf{K}
\def\1form{\omega}

\def\der{\partial_{t}}

\def\div{\mathrm{div}\,}
\def\trs{\mathrm{trace}\,}
\def\sff{\mathrm{II}}

\def\mcurvs{H}

\def\vecs{\mathfrak{T}}

\def\Hyp{\Sigma}

\def\fun{\rho}

\begin{document}

\title{Spacelike hypersurfaces in twisted product spacetimes with complete fiber and Calabi-Bernstein-type problems}

\author{Alberto Soria \\
Higher Technical School of Telecommunications Engineering, \\
Universidad Polit\'ecnica de Madrid  \\ 	
Av. Complutense 30, 28040 Madrid, Spain \\	
alberto.soria@upm.es}

\date{\today}

\maketitle

\begin{abstract}
In this article spacelike hypersurfaces immersed in twisted product spacetimes $I\times_f F$ with complete fiber are studied.
Several conditions ensuring global hyperbolicity are presented, as well as   
a relation that needs to hold on each spacelike hypersurface 
in $I\times_f F$ for it to be a simple warped product.
When the fiber is assumed to be closed (compact and without boundary) and the ambient spacetime has a suitable expanding behaviour, non-existence results for constant mean curvature hypersurfaces are obtained. Under the same hypothesis,  
a characterization of compact maximal hypersurfaces and other for 
totally umbilic ones with a suitable restriction on their mean curvature are presented.
The description of maximal hypersurfaces in twisted product spacetimes of the 
form $I\,{ }_{f}\!\!\times F$ with a one-dimensional Lorentzian fiber is also included. Finally, the mean curvature equation for a spacelike graph on the fiber is computed and as an application, some Calabi-Bernstein-type results are proven. We also include in an Appendix some known conformal geometry results 
describing the transformation of relevant tensors and operators under the action of a conformal map in a pseudo-Riemannian background.

\end{abstract}

\section{Introduction}\label{sec1}

Twisted and warped products are particular cases of pseudo-Riemannian manifolds which have interesting geometric properties and often have an important role in General Relativity (GR).
Indeed, warped products constitute a significant group of the exact solutions of the Einstein field equations. Among these is the Schwarzschild solution, which describes the outer geometry of the spacetime around massive bodies such planets, stars, black holes, etc. and the Robertson-Walker solutions, determining the geometry of a simply-connected expanding or contracting universe satisfying the cosmological 
principle of homogeneity and isotropy. In the area of geometry, twisted products have played an important role in the study of complex space forms \cite{LohnherrReckziegel1999}, in Lagrangian submanifolds \cite{Chenetal1998} and in curvature netted hypersurfaces \cite{Koike1995} among others. 
Let us first give the definition of pseudo-Riemannian twisted product manifold:

\begin{Def}
	Let $(B,g_B)$ and $(F,g_F)$ be pseudo-Riemannian manifolds. Let $f:B\times F\rightarrow (0,+\infty)$ be a positive smooth function. Consider the canonical projections $\pi_B:B\times F\rightarrow B$ and $\pi_F:B\times F\rightarrow F$. 
	Then the (simple) twisted product $B\times_f F$ of $(B,g_B)$ and $(F,g_F)$ is the differentiable manifold $B\times F$ endowed with the pseudo-Riemannian metric $\gb$ defined as
	\begin{equation*}
		\gb(X,Y)=g_B(d\pi_B(X),d\pi_B(Y))+f^2 g_F(d\pi_F(X),d\pi_F(Y))
	\end{equation*}
	for all vector fields $X,Y\in\mathfrak{X}(B\times F)$. In this setting, $(B,g_B)$ is the base manifold and $(F,g_F)$ the fiber manifold. 
\end{Def}
Twisted products were first introduced in \cite{Bishop1972} by Bishop as $\textit{umbilic products}$,
being so named because the leaves of the 
canonical foliation made up with copies of the fiber 
are totally umbilical in the ambient space.
Later on, Chen \cite{Chen1979,Chen1981} would refer to them as \textit{twisted products}, and were
generalized to double twisted products of two pseudo-Riemannian manifolds
by Ponge and Reckziegel in \cite{PongeReckziegel1993}. 
In case that $f$ is exclusively defined 
on the base $(B,g_B)$, twisted products reduce to 
the so-called warped products. Warped spacetimes with a Lorentzian open interval 
$(I,-dt^2)$ as base manifold and with isotropic and homogeneous fiber 
are called \textit{Robertson-Walker}. A natural extension for such spacetimes
was proposed for the first time 
in \cite{AliasRomeroSanchez1995}, and constitutes the family of      
\textit{Generalized Robertson-Walker} (GRW) spacetimes, where the fiber 
is not required to have constant sectional curvature. 
For a survey on GRW spacetimes, 
we refer the reader to \cite{ManticaMolinari2017survey}.
In this work we focus our interest on twisted products spacetimes with $(I,-dt^2)$ as base manifold and a complete Riemannian fiber $(F,g_F)$ of arbitrary dimension. However, some results will be obtained in twisted product spacetimes where the Lorentzian interval plays the role of the fiber.

It is possible to find in the literature several characterizations of twisted product pseudo-Riemannian manifolds.
As Ponge and Reckziegel show in \cite{PongeReckziegel1993},   
a pseudo-Riemannian product manifold $(M_1\times M_2,g)$ where the leaves of the associated canonical foliations $L_1=\{M_1\times\{q\},\, q\in M_2\}$ and $L_2=\{ \{p\}\times M_2,\, p\in M_1\}$
intersect perpendiculary everywhere has the structure of
a twisted product of the form $M_1\times_f M_2$ if and only if
the leaves of $L_1$ are totally geodesic and the leaves of $L_2$ totally umbilic in $M_1\times_f M_2$.
A local characterization for Riemannian and Lorentzian manifolds to be twisted products of the form $I\times_{f}F$ was also put forward by Chen in 
\cite{Chen2017}, where he proved that the existence of a so-called torqued vector field was equivalent to such local structure. This result was later on exploited by Mantica and Molinari in \cite{ManticaMolinari2017}, where any twisted spacetime is proven to be purely electric according to the Bel-Debever criterion and a condition for it to be a GRW spacetime is put forward. It is possible to find in the literature some other characterizations for twisted product manifolds to be a warped, like the one shown in \cite{FLopezGRioKupeliUnal2001}, where the authors find a necessary and sufficient condition
on the ambient Ricci tensor to be warped, and which they refer to as \textit{mixed Ricci-flatness}.
In Proposition \ref{twistedtowarped} of Section \ref{spacelikesection} we also find  a condition that needs to hold on immersed spacelike surfaces in a twisted product spacetime of the form 
$I\times_f F$ to be a truly GRW spacetime.

Spacelike hypersurfaces constitute a remarkable class of manifolds in the area of Riemannian and Lo\-rentzian geometry.  
Since Lichnerowicz's work \cite{Lichnerowicz1944}  
the study of maximal and constant mean curvature (CMC) hypersurfaces
has been a relevant problem in the context of differential geometry, partial differential equations, and has also played a fundamental role with regard to the dynamic aspects of GR. In addition, 
when a spacetime admits a foliation by CMC hypersurfaces there is a preferred choice for the time coordinate (read\cite{Rendall1996} for more details). 
It is well known that 
the whole nature of a spacetime can be determined by the geometry of spacelike hypersurfaces, as Choquet-Bruhat and Geroch's result \cite{ChoquetGeroch1969} about the Cauchy problem for the Einstein equations reveals. 
In particular, CMC and maximal hypersurfaces have been of great utility  
as initial data in the problem since the associated PDE system 
is reduced to a linear first order differential one and a nonlinear second order elliptic differential equation \cite{Lichnerowicz1944}.
Globally hyperbolic spacetimes do not always admit CMC Cauchy surfaces and the problem 
of determining them in a generic sense is still open. Several existence results for compact Cauchy surfaces
can be found in the literature, like the fundamental one put forward by Bartnik \cite{Bartnik1988}, which in turn has motivated  
others in the same line, 
like the one put forward by Galloway and Ling in \cite{GallowayLing2018}, where the existence 
of the CMC Cauchy surface is consequence
of an ambient curvature hypothesis related to the strong energy condition.     
Naturally, the study of CMC and maximal hypersurfaces in twisted product ambients 
is also an issue of remarkable interest.
The validity of many results in twisted product ambients depends on the monotonic 
properties of the associated twisted function. In this work several non-existence results of CMC hypersurfaces are obtained in $I\times_f F$ when the appropriate expanding conditions take place and $(F,g_F)$ is assumed to be a closed (compact and without boundary) Riemannian manifold. We also present a  
characterization of maximal spacelike hypersurfaces in twisted products of the form $I\times_f F$ and in those of type $I\,{ }_{f}\!\!\times F$ (i.e. with Lorentzian interval as fiber), and another one for     
 totally umbilical hypersurfaces with a suitable restriction on their mean curvature in $I\times_f F$.

Maximal surfaces occupy a prominent place among CMC ones.
As Brill and Flaherty remark in \cite{BrillFlaherty1976}, maximal hypersurfaces  
can be used to describe the transition between expansive and contractive phases in some 
significant models of the universe. An important result involving maximal hypersurfaces is the so-called Calabi-Bernstein theorem in the Minkowski spacetime, which establishes that the only entire hypersurfaces satisfying the maximal hypersurface differential equation are spacelike affine hyperplanes (see \cite{Rubio2017} for more details). This result was first obtained by Calabi \cite{Calabi1970} in the $(n+1)$-dimensional Minkowski spacetime $\mathcal{M}^{1,n}$ for $n\leq 3$.
The proof was based on the original Bernstein result for minimal surfaces in the three-dimensional Euclidean space. In \cite{ChengYau1976} Cheng and Yau proved the theorem for arbitrary dimension, which shows a different nature for the problem with respect to the Euclidean context, where the analogous result just holds up to ambient dimension $8$ (see \cite{BombieriGiorgiGiusti1969}). 
Since then many other related Calabi-Bernstein type results 
have been obtained in different ambient spaces. We refer the reader to \cite{AkamineHondaUmeharaYamada2020,AlbujerAlias2009,AledoRubio2015,AledoRubioSalamanca2017,AledoRubioSalamanca2019,
	BartoloCaponioPomponio2021,CaballeroRomeroRubio2011,CamargoCaminhadeLima2010,PelegrinRomeroRubio2017,PelegrinRomeroRubio2019note,PelegrinRomeroRubio2019,RomeroRubio2016} for more information. In the last section  we obtain  
results in $I\times_f F$ in the same line 
than the Calabi-Bernstein main one, where $(F,g_F)$ is assumed to be closed.

The paper is organized as follows: in Section \ref{Preliminaries} we first introduce our terminology and present 
general properties of twisted product spacetimes of the form $I\times_f F$. 
Section \ref{spacelikesection}
begins with the description of the causal character of spacelike hypersurfaces in $I\times_f F$ in terms of the 
fiber geometry, the twisted function and 
the so-called time-height function $\tau$ (see for instance \cite{MarsSoria2014}),
which defines the time-distance at which each point on the surface is located. In addition, a relation that needs to hold on each spacelike hypersurface immersed in $I\times_f F$ for it to be a GRW spacetime is presented.
We also address some global hyperbolicity questions 
in $I\times_f F$. Specifically, we prove that $I\times_f F$ 
turns out to be globally hyperbolic
whenever it admits a complete spacelike hypersurface where the restriction of $f$ is bounded. 
Section \ref{main} is devoted to present the main results of the work. 
Several non-existence results of CMC spacelike hypersurfaces in expanding or contracting twisted product spacetimes are put forward. Exploiting well-known properties of conformal maps,  
totally umbilical hypersurfaces with a suitable restriction on their mean curvature are characterized in $I\times_f F$, as well as all compact maximal hypersurfaces. 
Likewise, the analogous result for maximal hypersurfaces in twisted products of the form $I\,{ }_{\alpha}\!\!\times M$ is also presented, where the Lorentzian interval $(I,-dt^2)$ plays the role of the fiber manifold. 
In addition to assuming the compactness of $(F,g_F)$, we will also require 
$I\times_f F$ to be either    
a expanding or contracting background, or to admit a so-called transition hypersurface which delimits two different regions, an expanding one to its past and a contracting one to its future. 
Finally, in Section \ref{Calabisection} the expression for the mean curvature of a spacelike graph over the fiber in $I\times _f F$ is derived, and  some existence and uniqueness results for the associated maximal Calabi-Bernstein PDE equation are presented. At the end of the work we add an Appendix with elementary conformal geometry results which illustrate how the Levi-Civita connection, the shape operator, the mean curvature and the Laplacian on any non-degenerate submanifold transform under the action of a global conformal map in a pseudo-Riemannian background.

\section{Preliminaries}
\label{Preliminaries}

Consider  $(F,g_{F})$ and $(B,g_{B})$ two pseudo-Riemannian manifolds of arbitrary dimension. Let $f:B\times F\rightarrow \mathbb{R}$ be a positive smooth function. The  manifold $B\times F$ endowed with the metric

\begin{equation*}
	\gb=\pi^{*}_{B}(\gbase)+f^2\pi^{*}_{F}(g_{F}),
\end{equation*}
where $\pi_{B}$ and $\pi_{F}$ are the canonical projections onto the \emph{base} $(B,g_B)$ and the \emph{fiber} $(F,g_F)$ respectively, is known as twisted product manifold and it is  denoted as $B\times_f F$ (see \cite{PongeReckziegel1993}, for instance). As it is usual, vector fields $X\in \mathfrak{X}(B)$ and $V\in \mathfrak{X}(F)$ can be \textit{lifted} to vector fields $\overline{X}, \overline{V}$ in
$\mathfrak{X}(B \times F)$. We will denote then as $L(B) $ and $L(F)$ the sets of lifted vector fields from $\mathfrak{X}(B)$ and $\mathfrak{X}(F)$ respectively. For simplicity, we will generally not distinguish between $f \circ \pi_B$ and $f$ or between $X$ and $\overline{X}$, as it will be clear from
context. There exist useful formulae for evaluating the Levi-Civita connection on these fields (see \cite{PongeReckziegel1993}). Indeed, given $X, Y \in L(B)$ and $V, W \in L(F)$
and denoting by $\conb$, $\nabla^B$ and $\nabla^F$ the Levi-Civita connections on $B \times_f F$, $(B,g_B)$, and $(F,g_F)$ respectively, it is not difficult to see that
\begin{eqnarray}
	\conb_X Y&=&\conbase_X Y  \label{covderone}\\
	\conb_X V&=&\conb_V X=X(\log f)V \label{covdertwo}\\
	\conb_V W&=&\confiber_V W+V(\log f)W+W(\log f)V-\conb(\log f) \, \gb(V,W). \label{covderthree}
\end{eqnarray}

In this work we will consider a twisted product Lorentzian $(n+1)$-manifold constructed with a $n$-Riemannian manifold $(F,g_{F})$ and  an interval $I\subset \mathbb{R}$ endowed with the negative metric, i.e., $(I, -dt^2)$, where $(F,g_F)$ and $(I,-dt^2)$ can be fiber or base alternatively depending on the case. In case that the fiber is $(F,g_F)$, the twisted product spacetime will be denoted by $I\times_f F$, and its metric will be of the form $\gb=-dt^2+f^2 \pi^*_F(g_F)$.
Likewise, when the fiber is $(I,-dt^2)$ we will refer to it as $I\,{ }_{f}\!\!\times F$ and the metric is of type
$\gb=-f^2 dt^2+\pi^*_F(g_F)$. For convenience and for the sake of clarity, we will also use the convention
$\overline{g}=-dt^2+f^2 g_{F}$ and $\overline{g}=-f^2 dt^2+g_{F}$.
Note that when function $f$ is just defined on $(I,-dt^2)$, we restrict in the first case to the so-called Generalized Robertson-Walker (GRW) spacetimes 
(see \cite{AliasRomeroSanchez1995} for more details) and when $f$ is defined on $(F,g_F)$ to the standard static spacetimes (see \cite{PelegrinRomeroRubio2019}).

The study of twisted products is relevant because they constitute a much wider family  of manifolds with respect to the warped product one. Even though twisted products satisfying suitable curvature conditions are actually warped products as mentioned before, there are many of them which do not meet this property. Given a pseudo-Riemannian manifold $(M,g)$, a vector field $X\in\mathfrak{X}(M)$ is said to be a conformal Killing if there exists a smooth function $\mu:M\rightarrow \mathbb{R}$ so that $\mathcal{L}_{X}g=\mu g$. The following result illustrates the existence of twisted products with one-dimensional base which are not isometric to GRW spacetimes. To show it, we will exploit the existence of conformal Killing vector fields in warped products of the form $I\times_f F$.

\begin{Prop}
	There exist twisted product spacetimes which are not isometric to warped product ones.
\end{Prop}
\begin{proof}
	Let us consider $f:\mathbb{R}\rightarrow\mathbb{R}$ a non-constant periodic function of period $T$, differentiable and strictly positive. If $(t,x)$ stand for the usual coordinates in $\mathbb{R}^2$, $(\mathbb{R}^2,g_f)$ with    
	$g_f=-dt^2+f^2(t+x)dx^2$ is a twisted product Lorentzian manifold. 
	If we consider $\mathbb{S}^1=\mathbb{R}/T\mathbb{Z}$ where $T$ is the period of $f$, then
	$\mathbb{T}^2=\mathbb{S}^1\times\mathbb{S}^1$ is naturally endowed with the Lorentzian twisted metric
	$\tilde{g}_f=-dt^2+\tilde{f}dx^2$, where $\tilde{f}([t],[x]):=f(t+x)$.
	We know that for $f=1$,  $g_f$ is incomplete \cite{Sanchez1997} and hence 
	$(\mathbb{T}^2,\tilde{g}_f)$ is incomplete too. 
	In this case, there is no timelike conformal vector field on $(\mathbb{T}^2,\tilde{g}_f)$. In fact, if such vector field did exist,  $(\mathbb{T}^2,\tilde{g}_f)$ would be geodesically complete
	since every compact Lorentzian manifold aditting a conformal Killing vector is geodesically complete \cite{RomeroSanchez1995}. In particular, $(\mathbb{T}^2,\tilde{g}_f)$ can not be isometric to a warped product.
	The torus $(\mathbb{T}^2,\tilde{g}_f)$ can also be used to construct new examples of twisted products which do not admit 
	a warped product structure. Consider a compact Riemannian manifold $(F,g_F)$ and let us construct the warped product $\mathbb{T}^2\times_{\tilde{f}}F$. Since $(\mathbb{T}^2,\tilde{g}_f)$ is the base space of this warped product, it 
	is totally geodesic in $\mathbb{T}^2\times_{\tilde{f}}F$ when it is seen as $\mathbb{T}^2\times\{p_0\}$ with $p_0\in F$. Therefore, the fact that  $(\mathbb{T}^2,\tilde{g}_f)$ is incomplete tells us that $\mathbb{T}^2\times_{\tilde{f}}F$ is incomplete too.
	Note that $\mathbb{T}^2\times_{\tilde{f}}F$ is identical to the twisted product $\mathbb{S}^1\times_{\tilde{f}}(\mathbb{S}^1\times F)$ 
	whose metric is of the type $-dt^2+f^2(t+x)(dx^2+g_{\hat{F}})$ 
	on $\mathbb{R}\times(\mathbb{R}\times\hat{F})$, with $\hat{F}$ the 
	universal cover of $F$.
	Using again the main result of \cite{RomeroSanchez1995}, there is no timelike conformal vector field on $\mathbb{S}^1\times_{\tilde{f}}(\mathbb{S}^1\times F)$; in particular $\mathbb{S}^1\times_{\tilde{f}}(\mathbb{S}^1\times F)$ can not be isometric to a GRW spacetime. 	
\end{proof}

As we have just seen, the analysis of spacelike hypersurfaces in twisted product spacetimes can not be reduced to the existing one in GRW.
Let us first show general properties of interest of twisted product spacetimes of the form $I\times_f F$. The vector field  $\partial_t:=\frac{\partial}{\partial t}$ is defined globally on $I\times_f F$. Function $t$ is a smooth global time function and as consequence $I\times_f F$  is stably causal (see \cite{BeemEhrlichEasley}, pg. 64)
Furthermore, it is globally hyperbolic if
and only if its fiber is complete (Thm. 3.66 in \cite{BeemEhrlichEasley}  also applies to twisted product spacetimes). 
Besides, it is clearly time-orientable and consequently it constitutes a relativistic spacetime. In \cite{Chen2014} Chen characterizes 
GRW spacetimes by the presence of the so-called \textit{concircular vectors}, which were intruduced by Fialkow for the first time in \cite{Fialkow1939}. Specifically, a Lorentzian manifold $(M,g)$ of dimension greater or equal than three is a GRW spacetime if and only if there exists a timelike \textit{concircular vector} $K$, i.e. a timelike vector satisfying
\begin{equation*}
	\conb_X K=\mu X
\end{equation*}  
for all $X\in \mathfrak{X}(M)$, and for some smooth function $\mu: M\rightarrow \mathbb{R}$.
If we denote by $\overline{\nabla}$ the Levi-Civita connection of $I\times_f F$ and make use of the  
covariant derivative formulas (\ref{covderone}), (\ref{covdertwo}) and (\ref{covderthree}), we can see that for any vector field $X\in\mathfrak{X}(I\times_f F)$ and for $K:=f\partial _t$
we have
\begin{equation}
	\label{dervconf}
	\conb_X \vconf=df(X^F)\partial_t+(\der f) X=(\der f) X+\bm{\omega}(X) \vconf, 
\end{equation}  
where $X^F$ is the fiber-parallel component of $X$, and 
$\bm{\1form}$ is the one-form defined as 
\begin{equation}
	\label{oneform}
	\bm{\1form}(X)=\left( d(\log f)\circ \pi_F \right)(X),
\end{equation}
where $\pi_{_F}:I\times F\rightarrow F$ is the canonical projection onto $(F,g_F)$. In addition 
$\bm{\1form} (K)=0$, so $K$ is a \textit{torqued vector field} in the sense of \cite{Chen2017}, defined globally on $I\times_f F$,  and $\pmb{\omega}$
is defined to be the \textit{torqued one-form}. From ($\ref{dervconf}$) it is easy to prove
\
\begin{equation*}
	\mathcal{L}_{K}\gb=2(\der f)\gb+\bm{\1form}\otimes \bm{K}+\bm{K}\otimes\bm{\1form},
\end{equation*} 
where $\bm{K}$ is the $1$-form metrically equivalent to $K$.
In particular, when $U,V\in\mathfrak{X}(F)$, 
\begin{equation}\label{Lie}
	(\mathcal{L}_{K}\gb)(U,V)=f(\mathcal{L}_{\partial_t}\gb)(U,V)=2(\der f)\gb(U,V).
\end{equation}
Thus, $\partial_t$ is a spatially conformally stationary reference frame on $I\times_f F$ in the sense of \cite{RomeroSanchez1998}. The case where function $\partial_t f$ is constant signed deserves special interest, since the volume of any immersed spacelike submanifold which is Lie-dragged along the vector $\partial_t$ is monotonically altered.   
Indeed, the observers in $\partial_t$ detect expansion if $\der f\geq 0$ or compression when $\der f\leq 0$ along their proper times.

\section{Spacelike hypersurfaces in $\bm{I\times_f F}$}
\label{spacelikesection}

Given an $n$-dimensional connected manifold $\Sigma$, an immersion $\psi:\Sigma\longrightarrow I\times_f F$ is said to be spacelike if
the ambient Lorentzian metric $\gb$ induces, via $\psi$, a Riemannian metric $g_{_\Hyp}$ on $\Sigma$. In such case $(\Sigma,g_\Sigma)$
is called a spacelike hypersurface. In case that $(\Sigma,g_{_\Hyp})$ is orientable, there is a unique choice of a unit timelike vector field $N$ normal to $\Hyp$ with the same time-orientation than the vector
field $\partial_t$ (note that in this case $\gb(\partial_t, N)< 0$).  
The second fundamental form vector of $\Hyp$ is defined as $\vec{\sff}(X,Y)=(\conb_{X}Y)^{\perp}$, where $X, Y\in\mathfrak{X}(\Hyp)$ and  $Z^{\top}$ and $Z^{\perp}$ denote the tangent and orthogonal component of any vector field $Z$ along $\Hyp$ respectively. If we consider the Weingarten operator $A(X)=-\conb_X N$, we can express the second fundamental form vector as
\begin{equation*}
	\vec{\sff}(X,Y)=-\gb(A(X),Y)N, 
\end{equation*}
and consequently its covariant form reads $\sff(X,Y)=\gb(\vec{\sff}(X,Y),N)=\gb(A(X),Y)$. The mean curvature vector is defined by $\vec{H}=\frac{1}{n}\trs(\vec{\sff})$ and the mean curvature function $H=-\frac{1}{n}\trs(A)$. 
The Gauss equation relates
the ambient connection $\conb$ with the induced connection $\cons$ on $\Hyp$ and in this setting becomes
\begin{equation*}
	\conb_X Y=\cons_X Y-\gb(A(X),Y)N.
\end{equation*}
On the other hand, we  define on $\Hyp$ the hyperbolic angle $\theta:\Sigma\rightarrow\mathbb{R}$ as the function given by
\begin{equation}
	\label{hyperangle}	
	\cosh\theta=-\gb\left(N,\partial_{t}\right).
\end{equation}
The so-called ``time-height'' function $\tau=\pi_{I} \circ \psi:\Sigma\rightarrow\mathbb{R}$ (see for instance \cite{MarsSoria2014}) measures the time-distance at which each point of $\Sigma$ lies, where $\pi_I: I\times_f F\rightarrow I$ is the canonical projection onto the base $(I,-dt^2)$. Function $\tau$ is the restriction of $t$ to the hypersurface $\Sigma$, so the notation $\tau=t\vert_{\Sigma}$ will be sometimes considered. 
The expression for the induced metric $g_\Sigma$ on $\Sigma$ is given by
\begin{equation}
	\label{inducedmetricexpr}
	g_{\Sigma}=-d\tau\otimes d\tau+f^2 \pi_F^* (g_F).
\end{equation}  
With the following result we describe the causal character of spacelike hypersurfaces in 
twisted product spacetimes of the form $I\times_f F$ in terms $\tau$, $f$ and 
the fiber geometry. 
\begin{Prop}
	\label{spacelikecharac}
	Let $\psi:\Sigma\rightarrow I\times_{f} F$ be an oriented spacelike hypersurface and let
	$\tau=t\vert_{\Sigma}$ its associated time-height function. Then the following relations on the hypersurface $\Sigma$ are equivalent: 
	\begin{itemize}
		\item[(i)] $g_\Sigma$ is a Riemannian metric,
		\item[(ii)] $\vert d\tau(U) \vert \leq f\vert U^F\vert_{g_F}$ for all $U\in\mathfrak{X}(\Sigma)$,
		\item[(iii)] $\vert\nabla\tau\vert^2 \leq  f\vert\nabla\tau^F\vert_{g_F}$,
		\item[(iv)] $\gb(\partial_t,N)<0$,
	\end{itemize}
	where $\vert\,.\,\vert_{g_F}$ and $\vert\,.\,\vert$ are the norms associated to $g_F$ and $g_\Sigma$ respectively, and $N$ is the future unit normal to $\Sigma$.	
\end{Prop} 
\begin{proof}
	Let us first prove that $(i)$ and $(ii)$ are equivalent. Metric $g_\Sigma$ is Riemannian if and only if  
    $g_\Sigma(U,U)=-d\tau(U)^2+f^2 g_F(U^F,U^F)\geq 0$ for every $U\in\mathfrak{X}(\Sigma)$, which is equivalent 
	to $(ii)$. To prove that $(ii)$ implies $(iii)$, it suffices to consider $U=\nabla\tau$. For the equivalence of 
	$(i)$ and $(iii)$, if $g_\Sigma$ is Riemannian, then $g_\Sigma(\nabla\tau,\nabla\tau)\geq 0$, which implies $(iii)$. Reciprocally, given any $U\in \mathfrak{X}(\Sigma)$ orthogonal to $\nabla\tau$, $g_\Sigma(U,U)=f^2g_F(U^F,U^F)\geq 0$
	since $d\tau(U)=0$,
	so $g_F$ is a non-degenerate positive definite metric, i.e. Riemannian. Finally we prove the equivalence of $(i)$ and $(iv)$. The metric $g_\Sigma$ is Riemannian when each tangent plane $T_p \Sigma$ is endowed with a positive definite metric for all $p\in\Sigma$. This happens if and only if  
	the normal subspace $(T_p\Sigma)^{\bot}=\mathrm{span}\{N\}$ is timelike, i.e. $\gb(\partial_t,N)<0$ for both vectors 
	lie in the same future timelike cone.
\end{proof}

Any vector field $Z$ along a spacelike hypersurface $\psi: \Sigma\rightarrow I\times_f F$ decomposes into its tangent and normal components to $\Sigma$ as $Z=Z^{\top}+Z^{\bot}=Z^{\top}-\gb(Z,N)N$. On the other hand, its decomposition into its horizontal and vertical components reads
$Z=Z^0\partial_t+Z^F$,
where $Z^0=-\gb(\partial_t,Z)$. In the particular case that $Z\in\mathfrak{X}(\Sigma)$, $Z^0=d\tau(Z)$. One can easily check that $\cons \tau=-(\partial_t)^{\top}$. Squaring the decompositions $\partial_t=-\cons \tau-\gb(\partial_t,N)N$ and  
$N=-\gb( N, \partial_t)\partial_t+N^F$  we get 
\begin{equation}
	\label{gradrel}
	\vert\cons \tau\vert^2=\vert N^F \vert^2=-1+\gb(\partial_t,N)^2=\sinh^2\theta, 
\end{equation} 
where relation (\ref{hyperangle}) has been considered.
When $\psi: \Sigma\rightarrow I\times_f F$ is spacelike, (\ref{gradrel}) implies that $\gb(N,\partial_t)\leq -1$ 
along $\Sigma$. In the twisted product $I\times_f F$ the level hypersurfaces of the function $\pi_{I}$ constitute a distinguished family of spacelike hypersurfaces: the so-called spacelike slices. This family provides a foliation of the spacetime, whose leaves are the restspaces of the reference frame given by the vector field $\partial_t$, i.e. $\partial_t$ is the future unit normal vector field to each slice.
Every spacelike slice is totally umbilical $\cite{PongeReckziegel1993}$. In fact, the slice 
$\tau=t_0$ has $A=-\der (\log f)(t_0,p)\textrm{Id}$ as shape operator, where we denote by ``$\textrm{Id}$'' the identity endomorphism. As a consequence its  mean curvature function is 
$H=\der (\log f)(t_0,p)$. In contrast to the warped product case, the mean curvature of the slices is not necessarily constant along them.  
In case that $I\times_f F$ is globally hyperbolic, 
every spacelike slice is a Cauchy hypersurface.

As already mentioned, formula (\ref{Lie}) shows that when $\partial_tf\geq 0$ (resp. $\partial_t f\leq 0$) the observers in $\partial_t$ ({\it{comoving observers}}) measure a universe in expansion (resp. in contraction). It may also happen that $I\times_f F$ admits a maximal slice $S_{t_0}\equiv\{t= t_0\}$ separating an expanding 
phase to its past from a contracting one to its future (read \cite{BrillFlaherty1976} for more details).  
All these cases can be put together in the following definition: 
\begin{Def}
	Let $I\times_f F$ be twisted product spacetime. Then
	\begin{itemize}
		\item[(i)] if $\partial_t f>0$, we say that it is an \textit{expanding} twisted product spacetime.
		\item[(ii)] if $\partial_t f<0$, we say that it is a \textit{contracting} twisted product spacetime.
		\item[(iii)] if the spacetime is expanding for $t\leq t_0$ and it is contracting for $t\geq t_0$, where $t_0\in\mathbb{R}$ is constant, then the maximal hypersurface $S_{t_0}\equiv\{t= t_0\}$ is said be a \textit{transition spacelike slice}. 
	\end{itemize}
\end{Def}

Twisted product manifolds of the form $I\times_f F$ generalize GRW spacetimes. As Chen explains in \cite{Chen}, the existence of a torqued vector field on a pseudo-Riemannian manifold implies a local twisted product structure. Relation (\ref{dervconf}) shows that a torqued vector $K$ is concircular when the associated torqued one-form $\bm{\omega}=d(\log f)\circ\pi_F$ vanishes identically. By formula (\ref{oneform}), this happens if and only if $X^F(f)=0$ for any $X^F\in\mathfrak{X}(F)$, or equivalently
\begin{equation}
	\label{twistedtobewarped}
	\nabla^F f=0. 
\end{equation}
When this occurs, $f$ becomes constant
on each slice and the ambient spacetime is a GRW spacetime. 
In the following result we obtain a condition that is fulfilled on every spacelike hypersurface when $I\times_f F$ is actually a warped product.
\begin{Prop}
	\label{twistedtowarped}
	Let $\psi:\Sigma\rightarrow I\times_f F$ be an immersed spacelike hypersurface in the twisted product spacetime $I\times_f F$. In case that $I\times_f F$ is isometric to a warped product, the vector field  
	\begin{equation}
		\label{warpedhypcondition}
		\vecs=-(\der f)\cons\tau+\cons f
	\end{equation} 
	vanishes on $\Sigma$, where $\tau$ is the associated time-height function of $\Sigma$, and $\nabla$ its induced connection.
\end{Prop}
\begin{proof}
	Let $\psi:\Sigma\rightarrow I\times_f F$ be a spacelike hypersurface and $U\in\vfield(\Sigma)$. 
	The expression for the restriction to $\Sigma$ of the 
	torqued one-form $\bm{\1form}=d(\log f)\circ \pi_F$ associated to the vector $K=f\partial_t$ 
	reads
	\begin{equation}
		\label{1formsurface}
		\bm{\1form}(U)=\frac{df}{f}\left(U^F\right)=\frac{df}{f}\left(-d\tau(U)\partial_t+ U\right)=
		\frac{1}{f}\gb( -(\der f)\cons\tau+\cons f, U),
	\end{equation} 
	where the decomposition  
	$U=d\tau(U)\partial_t+U^F$ into its horizontal and vertical components
	has been considered. In case that
	$I\times_f F$ is a warped product the torqued one-form $\bm{\omega}$ is zero, which 
	occurs if and only if $\vecs=-(\der f)\cons\tau+\cons f=0$ along $\Sigma$.  
\end{proof}

\begin{remark}
	The vanishing of vector (\ref{warpedhypcondition}) on $\Sigma$ reduces to $\nabla f=0$ if $\Sigma$ is a slice, 
	so relation (\ref{twistedtobewarped}) is recovered since  
	$\nabla f=f^{-2}\nabla^F f$ in this case. 
\end{remark}

In general, 
the completeness of a twisted product of the form $B\times_f F$ can not be ensured 
even in case that the base manifold $(B,g_B)$ and the fiber $(F,g_F)$ are complete pseudo-Riemannian manifolds.
In this work we are interested in the study of twisted products $I\times_f F$ with complete Riemannian fiber $(F,g_F)$. As it happens in the warped case,  
$I\times_f F$ is globally hyperbolic under this hypothesis \cite{BeemEhrlichEasley}. We next prove that 
global hyperbolicity is also derived from the fact that $I\times_f F$ admits a complete spacelike hypersurface 
where the restriction of the twisted function $f$ is bounded. 
Some of the forthcoming results are an adaptation from \cite{AliasRomeroSanchez1995} to the twisted product spacetime context. To prove them we first need 
the following lemma, which establishes sufficient conditions for a map between Riemannian manifolds 
to be a covering one, and generalizes Lemma 3.3, Chapter 7 in \cite{DoCarmo}.
\begin{Lem}
	\label{DoCarmoextendedlemma}
	Let $(M,g_M)$ be a complete Riemannian manifold and let $h:(M,g_M)\rightarrow (N,g_N)$ be a local diffeomorphism 
	onto a Riemannian manifold $(N,g_N)$ where for all $p\in M$ and for all $U\in T_p M$, 
	\begin{equation}
		\label{covercondition}
		\vert U\vert_{g_M}\leq f\vert dh_p(U)\vert_{g_N}
	\end{equation}
	for some positive bounded function $f:M\rightarrow \mathbb{R}$, where $\vert\,.\,\vert_{g_M}$ and $\vert \,.\ \vert_{g_N}$ are the norms with respect to the metrics $g_M$ and $g_N$ respectively. Then $h$ is a covering map.
\end{Lem}
\begin{proof}
	As it is shown in \cite{DoCarmo1976} (pp. 383), this result would follow from the fact that 
	$h$ has the path lifting property for curves in $(N,g_N)$, i.e. given a smooth curve $c:[0,1]\rightarrow (N,g_N)$ and a point $q\in M$ with $h(q)=c(0)$, there is another curve $\overline{c}:[0,1]\rightarrow (M,g_M)$ starting at $q$ so that $h\circ \overline{c}=c$. Let us define the set
	\begin{equation*}
		A=\{a\in[0,1] \,\vert\, c:[0,a]\rightarrow (N,g_N) \,\,\text{can be lifted to} \,\, (M,g_M) \}.
	\end{equation*}
	Function $h$ is a local diffeomorphism at $q\in M$, which means that there are open neighbourhoods of 
	$U(q)$ and $V(c(0))$ so that $h:U(q)\rightarrow V(c(0))$ is a diffeomorphism. Then, there exists some $\epsilon>0$ so that $\overline{c}:[0,\epsilon]\rightarrow M$ satisfies $h\circ \overline{c}=c$. Hence $A\neq \emptyset$. Let us define $t_{0}:=\mathrm{sup}A$. We will first prove that any $c:[0,t_{0}]\rightarrow (N,g_N)$ with $t_{0}<1$ can not be lifted to $\overline{c}$. Otherwise, using again that $h$ is a local diffeomorphism, there would be open neighbourhoods $U(\overline{c}(t_{0}))$ and $V(c(t_{0}))$  such that $h:U(\overline{c}(t_{0}))\rightarrow V(c(t_{0}))$ is a diffeomorphism. This would imply the existence of some $t'>t_{0}$ such that $c:[0,t']\rightarrow (N,g_N)$  could be lifted to $\overline{c}$, which is an absurd by the definition of $t_{0}$. Therefore, the largest interval $I_{t_{0}}$ that allows the lift must be open in $[0,1]$, i.e. of the form $[0,t_{0})$. Let us next prove that $I_{t_{0}}$ is also closed by showing that $t_{0}\in I_{t_{0}}$ This way, $I_{t_{0}}=[0,1]$ for it is open and closed at the same time, and the lemma will be proven. To show that $t_{0}\in I_{t_{0}}$, consider an increasing sequence 
	$\{t_n\}_{n\in\mathbb{N}}\subset[0,t_{0})$ so that $\lim\limits_{n\rightarrow \infty}t_n=t_{0}$. The sequence
	$\{\overline{c}(t_n)\}_{n\in\mathbb{N}}$ must be contained in a compact set $K\subset M$. Otherwise, the sequence of distances $\{d(\overline{c}(t_n),\overline{c}(0))\}_{n\in\mathbb{N}}$ would not be bounded and the length $L_0^{t_n}(c)$ of $c$ between $t=0$ and $t=t_n$ would satisfy
	\begin{equation}
		\label{lengthrelation}
		L_0^{t_n}(c)=\int_{0}^{t_n}\vert c'(t) \vert_{g_N}dt=\int_0^{t_n}\vert dh_{\overline{c}(t)}(\overline{c}'(t))\vert_{g_N}\geq  \frac{1}{C}\int_0^{t_n}\vert\overline{c}'(t)\vert_{g_M}dt \geq \frac{1}{C}d(\overline{c}(t_n),\overline{c}(0)),
	\end{equation} 
	where in the two last inequalities we have applied that $\vert f \vert\leq C$ for some $C>0$ in combination 
	with hypothesis (\ref{covercondition}).
	Note that relation (\ref{lengthrelation}) implies the 
	unboundedness of $L_0^{t_{0}}(c)$ since $\{d(\overline{c}(t_n),\overline{c}(0))\}_{n\in\mathbb{N}}$ is not bounded, 
	which is an absurd.
	As a consequence, $\{\overline{c}(t_n)\}_{n\in\mathbb{N}}\subset K$, so it has an accumulation point $r\in M$. Consider a neighbourhood $U(r)$ of $r$ so that $h\vert_{U(r)}$ is a diffeomorfism. Then $c(t_{0})\in h(U(r))$, and by the continuity of $c$, there exists an interval $J=(t_{0}-\delta, t_{0}+\delta)$ for some $\delta>0$ such that $c(J)\subset h(U(r))$. Choose some $n_0\in \mathbb{N}$ so that for all
	$n\geq n_0$, $\overline{c}(t_{n})\in U(r)$ and consider the lift $\overline{g}$ of $c:J\rightarrow (N,g_N)$, which contains $r$. The lifts $\overline{g}$ and $\overline{c}$ coincide on $[0,t_{n})\cap J$ for $n\geq n_0$ since $h\vert_{U(r)}$ is bijective, and consequently on $[0,t_{0})\cap J$. Hence, $\overline{g}$ is an extension of $\overline{c}$ to $J$, and $\overline{c}$ is defined at $t_{0}$, which implies that $t_{0}\in I_{t_{0}}$.     	
\end{proof}

\begin{Lem}
	\label{Coveringmap}
	Let $\psi:\Sigma\rightarrow I\times_{f} F$
	be a spacelike hypersurface and $\pi_F\vert_{\Sigma}:\Sigma\rightarrow F$ its projection onto the fiber $(F,g_F)$. Then     
	$\pi_F\vert_{\Sigma}$ is a local diffeomorphism. In case that $\Sigma$ is complete and $f$ is bounded on 
	$\Sigma$, $\pi_F\vert_{\Sigma}:\Sigma\rightarrow F$ is a covering map.
\end{Lem}
\begin{proof}
	The squared norm of any $U\in\mathfrak{X}(\Sigma)$ 
	with respect to the hypersurface metric 
	(\ref{inducedmetricexpr}) satisfies
	\begin{equation}
		\label{squaredtangentvector}
		g_\Sigma(U,U)=-(d\tau(U))^2+f^2 g_F(U^F,U^F)\leq f^2 g_F(d\pi_F(U),d\pi_F(U))
	\end{equation}
	as a consequence of the decomposition $U=d\tau(U)\partial_t+U^F$.
	This relation implies that $\pi_F$ is locally injective.
	Otherwise there would be some non-zero $U\in\mathfrak{X}(\Sigma)$ such that $d\pi_F(U)=0$ at some point of it, 
	 which would mean that $g_\Sigma(U,U)=0$, contradicting the fact that $g_\Sigma$ is non-degenerate. 
	Besides, $\mathrm{dim}\,\Sigma=\mathrm{dim}\,F$, so $d\pi_F$ is locally bijective and hence a local diffeomorphism.
	On the other hand, the validity of (\ref{squaredtangentvector}) 
	makes the hypothesis 
	(\ref{covercondition}) hold true for the map
	$h:=\pi_F\vert_\Sigma:\Sigma\rightarrow F$. Since $\Sigma$ is
	complete and $f$ bounded,  $\pi_F$ is a covering map as a consequence 
	of Lemma \ref{DoCarmoextendedlemma}.
\end{proof}

A usual way to prove that a Riemannian manifold is complete is by considering divergent curves, i.e. those ones that can not be enclosed by any compact set. The Hopf-Rinow theorem (see \cite{DoCarmo} for instance) entails that a Riemannian manifold containing a divergent curve of finite length can not be complete. We exploit this fact  
to prove the following completeness results for spacelike hypersurfaces in $I\times_f F$:
\begin{Prop}
	Let $I\times_{f} F$ be a twisted product spacetime and assume that it admits a complete immersed spacelike hypersurface
	$\psi:\Sigma\rightarrow I\times_{f} F$ where $f$ is bounded. Then the fiber $(F,g_F)$ is complete.
\end{Prop}
\begin{proof}
	Given any $U\in\mathfrak{X}(\Sigma)$, the inequality   
	\begin{equation}
		\label{normboundcondition}
		\vert U\vert_{g_\Sigma}\leq C\vert d\pi_F(U)\vert_{g_F}
	\end{equation}
	follows from (\ref{squaredtangentvector}), where $C$ is any upper bound for $f$.
	The hypothesis of 
	Lemma \ref{Coveringmap} are satisfied, so the map $\pi_F\vert_\Sigma: \Sigma\rightarrow F$ is a covering map. In particular this means that the path lifting property for curves in $(F,g_F)$ is satisfied in this context. Hence, any divergent curve 	
	$c:[0,+\infty)\rightarrow (F,g_F)$ can be lifted to a divergent $\overline{c}:[0,\infty)\rightarrow \Sigma$ so that $c=\pi_F\circ \overline{c}$. 	
	By (\ref{normboundcondition}), the lengths of $c$ and $\overline{c}$ satisfy
	\begin{equation}
		\label{completenessrelation}
		L_0^\infty(c)=\int_0^{\infty}\vert c'(t)\vert_{g_F}dt=\int_0^{\infty}\vert d\pi_F(\overline{c}'(t))\vert_{g_F}dt\geq \frac{1}{C}\int_0^{\infty}\vert\overline{c}'(t)\vert_{g_\Sigma}dt=\frac{1}{C}L_0^\infty(\overline{c}).
	\end{equation}
	The right-hand side of (\ref{completenessrelation}) is infinite since $\overline{c}$ is a divergent curve in a complete manifold, which in particular implies that $c$ also has infinite length in $(F,g_F)$. Since this happens to any arbitrary divergent curve $c$ in the fiber, $(F,g_F)$ is complete too. 
\end{proof}

\begin{Prop}
	Let $I\times_f F$ be a twisted product spacetime. The following assertions hold true:
	\begin{itemize}
		\item[(i)] If $I\times_{f} F$ admits a compact spacelike hypersurface $\psi:\Sigma\rightarrow I\times_{f} F$, then the fiber $(F,g_F)$ is compact.
		\item[(ii)] If the universal cover of the fiber is compact and $\psi:\Sigma\rightarrow I\times_{f} F$ is any complete spacelike hypersurface on which $f$ is bounded, then $\psi:\Sigma\rightarrow I\times_{f} F$ is compact.
	\end{itemize}
\end{Prop}
\begin{proof}
	Compact Riemannian manifolds are complete. Since $f$ is bounded on 
	$\Sigma$, we apply Lemma \ref{Coveringmap} and $\pi_F\vert_{\Sigma}$ is a covering map. Hence $F=\pi_F(\Sigma)$ is compact and that proves $\it{(i)}$.         
	On the other hand, 
	the assumptions of $\it{(ii)}$ also imply that $\pi_F\vert_{\Sigma}$ is a covering map. Since the universal cover of the fiber is compact by hypothesis, both $F$ and  $\Sigma=(\pi_F\vert_{\Sigma})^{-1}(F)$ are compact too. 
\end{proof}

\section{Main results}
\label{main}

We have already mentioned the big interest that $CMC$ and maximal hypersurfaces have in several areas of mathematics and
physics. In the current section results involving CMC, maximal, totally geodesic
and totally umbilical  spacelike hypersurfaces in twisted products of the form $I\,{ }_{f}\!\!\times M$ and $I\times _f M$ are presented.
For the validity of many of them
an appropriate expanding behaviour and the compactness of the fiber 
become essential hypothesis.  
The geometry of the surface is encoded in the time-height function $\tau$, which
measures the time-distance of any point lying on it.  
This motivates 
the computation of the Laplacian of $\tau$ on such hypersurfaces as a fundamental tool to characterize these families of manifolds. In the next result we obtain the expression for $\triangle\tau$, where $\triangle$ denotes the Laplacian on the hypersurface with respect to its induced metric. 
\begin{Prop}
	\label{Proplaptau}
	Let $\psi:\Sigma\rightarrow I\times_{f} F$ be an immersed spacelike hypersurface in $I\times_f F$, and $\tau$ its corresponding time-height function. The expression for the Laplacian of $\tau$ with respect to the induced metric on $\Sigma$ is
	\begin{equation}
		\label{laptau}
		\triangle\tau=-(n+\vert\cons\tau\vert^2)\der \log f-n\mcurvs\gb(\partial_t,N),
	\end{equation}
	where $H$ is the mean curvature of $\Sigma$ along its future unit normal $N$.
\end{Prop}
\begin{proof}
	
	The Laplacian of $\tau$ can be obtained considering the decomposition of the torqued vector $\vconf=f\partial_t$ into 
	its tangential and normal components to the hypersurface $\psi:\Sigma\rightarrow I\times_{f} F$,
	\begin{equation}
		\label{vconfsplit}
		\vconf=\vconf^{\top}-\gb(\vconf,N)N=-f\cons \tau-\gb(\vconf,N)N.
	\end{equation}
	Let $U\in\mathfrak{X}(\Sigma)$. The ambient covariant derivative of (\ref{vconfsplit}) along $U$ reads
	\begin{equation*}
		\conb_U\vconf=\conb_U\left(-f\cons\tau\right)-\conb_U\left(\gb(\vconf,N)N\right).
	\end{equation*}  
	 Using formula (\ref{dervconf}) 
	 in the left-hand side of the above expression  
	 and applying the Leibniz rule on the right-hand side, it becomes 
	\begin{equation*}
		(\der f)U+\1form(U)\vconf=-U(f)\cons\tau-f\conb_U\left(\cons\tau\right)-U(\gb(\vconf,N))N-\gb(\vconf,N)\conb_U N.
	\end{equation*} 
    Inserting relation 
	(\ref{1formsurface}) for $\1form(U)$ and contracting the whole expression with any other $V\in\mathfrak{X}(\Sigma)$ gives
	\begin{equation*}
		\left(\gb(U,V)+\left(\cons_U\tau\right)\left(\cons_V\tau\right)\right)(\der f)=-f\gb\left(\cons_U\left(\cons\tau\right),V\right)+\gb(\vconf,N)\sff(U,V), 
	\end{equation*}
	where $II$ is the second fundamental form tensor of $\Sigma$. The trace of the above expression gives the stated relation for the Laplacian, after solving for $\triangle\tau$. 
	
\end{proof}

Formula (\ref{laptau}) can be applied to obtain rigidity results for slices in $I\times_f F$. In the following proposition,
we prove that a spacelike hypersurface $\psi:\Hyp\rightarrow I\times_f F$ is necessarily a slice provided   
its mean curvature satisfies a suitable bound involving $\partial_t f$, where such derivative is required to be constant signed along $\Sigma$.
 \begin{Prop}
	Let $\psi:\Hyp\rightarrow I\times_f F$ be a compact spacelike hypersurface in a twisted product spacetime $I\times_f F$. 
	If 
	\begin{itemize}
		\item[(i)] $\partial_t f\geq 0$ on $\Sigma$ and $H\geq \frac{\der (\log f)}{\gb(\partial_t,N)}$,
		\item[(ii)] or $\partial_t f\leq 0$ on $\Sigma$ and $H\leq \frac{\der (\log f)}{\gb(\partial_t,N)}$, 
	\end{itemize} 
	then $\Sigma$ is a slice.
\end{Prop}
\begin{proof}
	Note that formula ($\ref{laptau}$) can also be written as 
	\begin{equation*}
		\triangle\tau=-n\left(\der (\log f)+H\gb(\partial_t,N) \right)-\der (\log f)\vert\nabla\tau\vert^2.
	\end{equation*}
	Condition $H\geq \frac{\der (\log f)}{\gb(\partial_t,N)}$ (resp. $H\leq \frac{\der (\log f)}{\gb(\partial_t,N)}$) is equivalent to $\der (\log f)+H\gb(\partial_t,N)\geq 0$ (resp. $\der (\log f)+H\gb(\partial_t,N)\leq 0$). Hence, by (\ref{laptau}) the inequality $\triangle\tau\leq 0$ (resp. $\triangle\tau\geq 0$) holds. Therefore $\tau$ is constant in both cases $(i)$ and $(ii)$ as a consequence of the compactness hypothesis on $\Sigma$.
\end{proof}

The following non-existence results of CMC hypersurfaces in $I\times_f F$ 
are also derived from relation (\ref{laptau}), where an appropriate expanding 
behaviour of the ambient space is required:
\begin{Tma}
	The following assertions hold true in expanding and contracting twisted product spacetimes:
	\begin{itemize}
		\item[(i)] There are no compact spacelike hypersurfaces
		with negative CMC in expanding $I\times_f F$.
		\item[(ii)] There are no compact spacelike hypersurfaces
		with positive CMC in contracting $I\times_f F$.
	\end{itemize}
\end{Tma}
\begin{proof}
	Let $\psi:\Sigma\rightarrow I\times_f F$ be a compact CMC spacelike hypersurface. Hence
	$\tau$ reaches a global maximum at $(t^0,x^0)$ and minimum at $(t_0,x_0)$ on $\Sigma$. As a consequence, $\triangle \tau\vert_{(t_0,x_0)} \geq 0$ and $\triangle \tau\vert_{(t^0,x^0)} \leq 0$. Using the Laplacian expression (\ref{laptau}), we obtain
	that the following inequality holds at $(t_0,x_0)$: 
	\begin{equation*}
		0\leq \triangle \tau\vert_{(t_0,x_0)}=-n\partial_t(\log f)\vert_{(t_0,x_0)}+nH,
	\end{equation*}
	which means 
	\begin{equation}
		\label{Hlowerbound}	
		H\geq \partial_t(\log f)\vert_{(t_0,x_0)}.
	\end{equation}
	This lower bound for $H$ implies that in an expanding twisted product spacetime case, there are no compact spacelike hypersurfaces with negative constant mean curvature.
	An analogous reasoning for the maximum $(t^0,x^0)$ leads to 
	\begin{equation}
		\label{Hupperbound}	
		H\leq \partial_t(\log f)\vert_{(t^0,x^0)},
	\end{equation}
	which also proves that neither are there any positive CMC spacelike hypersurfaces in a contracting background.
\end{proof}
\begin{Obs}
	Note that for the validity of inequalities (\ref{Hlowerbound}) and (\ref{Hupperbound}) it is sufficient for  
	the ambient spacetime to have a twisted product structure in a neighbourhood of $(t_0,x_0)$ and $(t^0,x^0)$. Such inequalities are optimal, since slices $\Sigma_{t_0}\equiv\{t=t_0\}$ in GRW spacetimes are CMC hypersurfaces with $H=\partial_t(\log f)(t_0)$.    
\end{Obs}

Finding conditions for the Laplacian of $\tau$ to be constant signed 
is an effective approach to obtain rigidity results for slices 
in twisted product spacetimes.
In some occasions, the sign of $\triangle \tau$ is not clearly determined 
under generic geometric conditions. Conformal transformations are useful maps which allow to modify the geometry of manifolds in a very precise manner and 
can help to discern under which hypothesis this sign is well-defined. In twisted product spacetimes, slices belong to the family of totally umbilical hypersurfaces, which are those whose shape operator is proportional to the identity endomorphism. As it is shown  in Proposition \ref{Propconformaltransf} in the Appendix, totally umbilic submanifolds are preserved under the action of conformal transformations.
If the twisted product spacetime $\Mb=I\times_f F$ undergoes the conformal
transformation defined by $\gone=e^{2\phi}\gb$ with $\phi=-\log f$, then $\gone$ becomes  $\gone=-\alpha^{2} dt^2+g_{F}$, where $\alpha:=f^{-1}$,
and $\overline{M}$ acquires a twisted structure of the form $I\,{ }_{\alpha}\!\!\times F$, where this time  $(I,-dt^2)$ is the fiber. When the metric is of the form $\gone=-\alpha^2 dt^2+g_{F}$, slices in $(\Mb,\gone)$  are totally geodesic \cite{PongeReckziegel1993}. Indeed, the future timelike unit normal of 
the slice $\Sigma_{t_0}\equiv\{t=t_0\}$ is $\alpha^{-1}\partial_t$, so  
given any $X,Y \in \vfield(\Sigma_{t_0})$, the second fundamental form $II$ of $\Sigma_{t_0}$ is
\begin{eqnarray*}
	II(X,Y)&=&-\gone(\tilde{\nabla}_X\left(\alpha^{-1}\partial_t\right),Y)=-\gone(X(\alpha^{-1})\partial_t+\alpha^{-1}\tilde{\nabla}_X\partial_t,Y)=-\alpha^{-1}\gone(\tilde{\nabla}_X\partial_t,Y)=0 
\end{eqnarray*}
since $\tilde{\nabla}_X \partial_t=X(\log\alpha)\partial_t$.

In this work we exploit some useful properties of conformal maps 
to obtain characterizations of slices in both sorts of twisted product spacetimes $I\times_f F$ and $I\,{ }_{\alpha}\!\!\times F$ provided $f$ and $\alpha$ enjoy suitable monotonic behaviour with respect to $t$. We next compute the Laplacian of the time-height function of any spacelike surface 
in $I\,{ }_{\alpha}\!\!\times F$. 
\begin{Lem}
	\label{ProptaulaplacianMtilde}
	Let  $\psi: \Hyp\rightarrow (\Mb=I\,{ }_{\alpha}\!\!\times F,\gone)$ an immersed spacelike hypersurface
	in $I\,{ }_{\alpha}\!\!\times F$. 
	The Laplacian $\tilde{\triangle}\tau$ with respect to the induced metric on $\Sigma$
	reads
	\begin{equation}
		\label{taulaplacianMtilde}
		\tilde{\triangle}\tau=\alpha^{-2}\left( \left(1+\cosh^2\theta \right)\der\log\alpha+n\tilde{H}\alpha \cosh\theta-2\alpha\cosh\theta \tilde{N}(\log\alpha)   \right), 
	\end{equation}
	where $\tilde{H}$ is the mean curvature of $\Sigma$ along the $\tilde{g}$-unit normal $\tilde{N}$ and   $\cosh\theta=-\gone\left(\tilde{N},{\alpha}^{-1}\partial_{t}\right)$ is the hyperbolic cosine of $\theta$.
	\end {Lem}
	\begin{proof}
		The metric $\gone=-\alpha^2 dt^2+g_F$ is obtained from the twisted metric $\gb=-dt^2+f^2 g_F$ after applying the conformal transformation $\gone=\alpha^2\gb$,
		where $\alpha:=f^{-1}$. 
		In this case expression (\ref{conformallaplacianrelation}) in the Appendix 
		establishes that the respective Laplacians $\triangle$ and $\tilde{\triangle}$ associated to 
		$\gb$ and $\gone$ satisfy
		\begin{equation}
			\label{conflaplacianrelation}
			\tilde{\triangle}\tau=\alpha^{-2}(\triangle \tau+(n-2)\gb(\nabla \log\alpha, \nabla \tau)).
		\end{equation}	
		Since $\vert\triangle \tau\vert^2=\cosh^2{\theta}-1$ and the relation for the mean curvatures $H$ and $\tilde{H}$ associated to the metrics $\gb$ and $\gone$ is $H=(\tilde{H}-\tilde{N}(\log \alpha))$ by formula (\ref{confmeancurvature}) in the Appendix, 
		 expression (\ref{laptau}) for $\triangle \tau$ on $\psi:\Sigma\rightarrow I\times_f F$ can be rewritten as 
		\begin{eqnarray}
			\label{laptautilde}
			\triangle\tau=(n-1+\cosh^2{\theta})\partial_t \log\alpha+n\alpha(\tilde{H}-\tilde{N}(\log\alpha))\cosh \theta.
		\end{eqnarray} 	
		On the other hand, it is immediate to see that $\conone\tau=-\alpha^{-2}(\partial_t)^{\top}$. The decomposition of the $\gone$-unit vector $\alpha^{-1}\partial_t$ into its tangent and normal component to $\Sigma$ imply 
		\begin{equation}
			\label{gradtildetau}
			\tilde{\nabla}\tau=\alpha^{-1}\left(-\alpha^{-1}\partial_t-\gone(\alpha^{-1}\partial_t,\tilde{N})\tilde{N}\right).
		\end{equation}
		Likewise, the decomposition of the ambient gradient $\tilde{\nabla}^{\gone}\log\alpha$ in $I\,{ }_{\alpha}\!\!\times M$  gives
		\begin{equation}
			\label{gradtildelogalpha}
			\tilde{\nabla}\log\alpha=\tilde{\nabla}^{\gone}\log\alpha+\tilde{N}(\log\alpha)\tilde{N}.
		\end{equation}  
		Considering (\ref{gradtildetau}) and (\ref{gradtildelogalpha}), we obtain
		\begin{equation}
			\label{productgradients}
			\gb(\nabla \log\alpha, \nabla \tau)=\alpha^2\gone(\tilde{\nabla}\log\alpha,\tilde{\nabla}\tau)=-\partial_t\log\alpha+\alpha\cosh\theta \tilde{N}(\log\alpha).
		\end{equation}
		Finally,  (\ref{taulaplacianMtilde}) follows after plugging (\ref{laptautilde}) and (\ref{productgradients}) into (\ref{conflaplacianrelation}).

	\end{proof}
	
	A natural question that arises is whether slices are the only maximal hypersurfaces in $I\,{ }_{\alpha}\!\!\times F$ when appropriate expanding conditions are considered.
	With the help of conformal maps, we prove they are indeed. To this purpose, we will first need the following lemma:  
	\begin{Lem}
		\label{Propconftransf2}
		Let $I\,{ }_{\alpha}\!\!\times F$ be a twisted product spacetime, and $\psi: \Hyp\rightarrow I\,{ }_{\alpha}\!\!\times F$ an immersed maximal compact spacelike hypersurface, where $\mathrm{dim}\,F> 2$. 
		Assume that $\der\alpha$ is a constant signed function. Then there exists a conformal transformation under which the  Laplacian of the time-height function $\tau$ of $\Sigma$ is constant signed.   
		\end {Lem}
		\begin{proof}
			For any $p\in \mathbb{R}$, the Laplacians $\tilde{\triangle}$ and $\hat{\triangle}$ of the conformally related metrics $\gone=-\alpha^2 dt^2+g_F$ of $I\,{ }_{\alpha}\!\!\times F$ and   
			$\gtwo=\alpha^{2p}\gone$ satisfy
			\begin{equation*}
				\hat{\triangle}\tau=\alpha^{-2p}(\tilde{\triangle}\tau+(n-2)\gone(\conone\log(\alpha^p),\conone\tau))=\alpha^{-2p}(\tilde{\triangle}\tau-\alpha^{-2}p(n-2)
				\gone(\conone\log \alpha,\conone\tau))
			\end{equation*}
			by virtue of formula (\ref{conformallaplacianrelation}) in the Appendix. Particularizing expression (\ref{taulaplacianMtilde}) for a maximal hypersurface in $(I\,{ }_{\alpha}\!\!\times F,\gone)$ and substituting in the above relation gives
			\begin{equation*}
				\hat{\triangle}\tau=\alpha^{-2p-2}\left( (1-p(n-2)+\cosh^2\theta)\der(\log\alpha)+(-2+p(n-2))\alpha(\cosh\theta)\norone(\log\alpha)  \right).	
			\end{equation*} 
			Note that the second term in the right-hand side of this expression vanishes for $p=\frac{2}{n-2}$ (we assume $n>2$). For this value of $p$ the relation between the Laplacians turns out to be
			\begin{equation}
				\label{laplacianconftrasnf3}
				\hat{\triangle}\tau=\alpha^{\frac{-2n}{n-2}}(\sinh^2\theta)\der\log\alpha=\left(\alpha^{\frac{-n}{n-2}}\right)^{2}(\sinh^2\theta)\frac{\der\alpha}{\alpha}, 
			\end{equation} 
			which clearly establishes the statement of the lemma provided $n>2$.
		\end{proof}

		\begin{Obs}
			Note that the above conformal transformation is not well defined for $n=2$.
			However, this problem can be addressed in the same way as in 
			\cite{PelegrinRomeroRubio2019} (Remark 1). Specifically,
			given $I\,{ }_{\alpha}\!\!\times F$
			we can consider
			the higher-dimensional ambient space $I\,{ }_{\alpha}\!\!\times \mathbb{S}^2\times F$, and $\mathbb{S}^2\times \Sigma$ as immersed spacelike hypersurface, whose mean curvature vector is the same than the one of $\Sigma$ lifted to the extended spacetime. The conformal transformation can be performed in this new background.
		\end{Obs}
		
		The characterization of spacelike maximal hypersurfaces in $I\,{ }_{\alpha}\!\!\times F$ follows as a consequence of the above lemma. 
		\begin{Tma}
			\label{maximalslice}
			Let $(\Mb=I\,{ }_{\alpha}\!\!\times F,\gone)$ be a twisted product spacetime, where $(F,g_F)$ is compact.
			Suppose that either
			\begin{itemize}
				\item[(i)] $\partial_t\alpha$ is constant signed,
				\item[(ii)] or a slice $\Sigma_{t_0}\equiv\{t=t_0\}$
				divides $I\,{ }_{\alpha}\!\!\times F$ in two different regions,  
				with $\partial_t \alpha<0$ if $t\leq t_0$ (past region), and $\partial_t \alpha>0$ if $t\geq t_0$ (future region).   
			\end{itemize}  
			Then, an immersed compact hypersurface $\psi: \Hyp\rightarrow I\,{ }_{\alpha}\!\!\times F$ is maximal if and only if it is a slice.   
		\end{Tma}
		\begin{proof}
			Let $\psi:\Sigma\rightarrow I\times_f F$ be a maximal hypersurface in $I\,{ }_{\alpha}\!\!\times F$. 
			Lemma \ref{Propconftransf2} can be applied under the assumptions of $(i)$, 
			so $\Hyp$ becomes a compact hypersurface in $(\Mb,\gtwo=\alpha^{\frac{4}{n-2}}\gone)$
			where $\hat{\triangle}\tau$ is constant signed, and hence  
			$\psi: \Hyp\rightarrow I\,{ }_{\alpha}\!\!\times F$ is  
			a slice in $I\,{ }_{\alpha}\!\!\times F$.      
			To prove $(ii)$, 
			let us consider the Laplacian of
			the  function $(\tau-t_0)^2$ in  $(\Mb,\gtwo=\alpha^{\frac{4}{n-2}}\gone)$
			\begin{equation}
				\hat{\triangle} (\tau-t_0)^2=2\left(\vert\hat{\nabla}\tau\vert^2+(\tau-t_0)\hat{\triangle}\tau\right).    
			\end{equation}
			Given any $p\in\Sigma$ in the future region where $\partial_t \alpha >0$ and $\tau-t_0\geq 0$,  
			$\hat{\triangle}\tau\geq 0$ by
			virtue of (\ref{laplacianconftrasnf3}), which makes $\hat{\triangle} (\tau-t_0)^2$
			non-negative. 
			An analogous argument for any $p\in\Sigma$ in the past region where $\partial_t \alpha <0$ and $\tau-t_0\leq 0$ 
			entails that $\hat{\triangle} (\tau-t_0)^2\geq 0$ too. In any case,
			$\hat{\triangle} (\tau-t_0)^2\geq 0$, so $\tau$ must be constant as a consequence of the compactness of $\Sigma$.
		\end{proof}

		The results obtained so far allow us to find a characterization of the totally umbilical hypersurfaces in $I\times_f F$ with a suitable restriction on their mean curvature.
		\begin{Tma}
			\label{umbilicaltransition}
			Let $I\times_f F$ be a twisted product spacetime with compact fiber, that is either 
			\begin{itemize}
				\item[(i)] expanding (or contracting),
				\item[(ii)] or  it admits a transition slice $S_{t_0}\equiv\{t= t_0\}$.  
			\end{itemize} 
			Then the only totally umbilical compact hypersurfaces  whose mean curvature satisfies the equation $H=\gb(N,\conb\log f)$ in $I\times_f F$ are the slices. 
		\end{Tma}
		\begin{proof}
			Let $\psi:\Hyp\rightarrow I\times_f F$ be a totally umbilical compact hypersurface whose 
			mean curvature satisfies $H=\gb(N,\conb\log f)$. Case $(i)$ can be addressed applying 
			Corollary \ref{charactmaximal} in the Appendix, which gives the maximality of 
			$\Hyp$ in $I\,{ }_{f^{-1}}\!\!\times F$. Since the first time derivative of $\alpha:={f^{-1}}$ is constant signed, $\Hyp$ has to be necessarily a slice in $I\,{ }_{f^{-1}}\!\!\times F$ by Theorem \ref{maximalslice}, and consequently also in  $I\times_f F$. The proof of $(ii)$ follows the same argument as Theorem \ref{maximalslice}, $(ii)$. Specifically,  
			$\hat{\triangle} (\tau-t_0)^2$ is proven to be non-negative 
			regardless of the region where any point of $\Hyp$ lies, 
			and by the compactness assumption on the hypersurface, $\tau$ is constant. 
		\end{proof}

		\begin{Obs}
			The above theorem is optimal in the sense that it does not hold provided the signs of $\partial_t f$
			in the future and past region of the transition slice $S_{t_0}$ are opposite to the ones stated, as it occurs 
			in the De Sitter spacetime.  
		\end{Obs}

		We finally conclude the section with the characterization of compact maximal  
		spacelike hypersurfaces in $I\times_f F$. To obtain this result
		we first need to prove the following lemma, which establishes that any compact spacelike hypersurface
		must be a totally geodesic slice provided 
		its mean curvature and the restriction 
		of the time-derivative of $f$ to it have opposite signs. 
		\begin{Lem}
			\label{signLaplacian}
			Let $\psi:\Hyp\rightarrow I\times_f F$ be a compact spacelike hypersurface in the twisted product spacetime
			$I\times_f F$ and $H$ its mean curvature.
			If $(\der f)\cdot H\leq 0$ on $\Hyp$,
			then it is a totally geodesic spacelike slice. 
		\end{Lem}
		\begin{proof}
			Let $\psi:\Hyp\rightarrow I\times_f F$ be a compact spacelike hypersurface in $I\times_f F$.
			The expression for $\triangle\tau$ in (\ref{laptau}) becomes constant signed provided $\der f$ and $H$ have opposite signs, 
			so $\tau$ is constant as a consequence of the compactness of $\Sigma$. 
			Since the terms on the right hand-side of ($\ref{laptau}$) have the same sign both must be zero, which in particular implies the maximality of $\Sigma$. Any slice in $I\times_f M$ is totally umbilical, so $\Sigma$ turns out to be totally geodesic too.
		\end{proof}

		\begin{Tma}
			\label{CharactMaximal1}
			Let $I\times_f F$ be a twisted product spacetime with compact fiber, that is either	
			\begin{itemize}
				\item[(i)]  expanding (resp. contracting),
				\item[(ii)] or it admits a transition spacelike slice $S_{t_0}\equiv\{t= t_0\}$.
			\end{itemize} 
			Then every compact maximal hypersurface  $\psi:\Hyp\rightarrow I\times_f F$ is also a totally geodesic spacelike slice. If  $(ii)$ happens, $S_{t_0}$ is the only compact maximal hypersurface in $I\times_f F$. 
		\end{Tma}
		\begin{proof}
			Let $\psi:\Hyp\rightarrow I\times_f F$ be a compact maximal hypersurface in $I\times_f F$. 
			Under the assumptions of $(i)$, $(\der f)\cdot H\leq 0$ holds on $\Hyp$, so
			$\Sigma$ is totally geodesic as a consequence of Lemma \ref{signLaplacian}. 
			To prove $(ii)$, it is enough to see again that
			$\triangle (\tau-t_0)^2$ is non-negative, which implies that $\tau=t_1$ for some $t_1\in \mathbb{R}$. If $t_1>t_0$, the mean curvature of $\Sigma$ would be strictly negative, whereas if $t_1<t_0$ it would be strictly positive, which contradicts the maximality hypothesis in both cases. The only possibility for the slice $\Sigma_{t_1}\equiv\{t=t_1\}$ to be maximal 
			is that $t_1=t_0$, i.e when it is the transition slice.
		\end{proof}

		\begin{remark}
			It would also be natural to consider twisted product spacetimes admitting a maximal spacelike slice $S_{t_0}\equiv\{t= t_0\}$ such that they are contracting for $t\leq t_0$ and expanding for $t\geq t_0$. Nevertheless, in this case it is not possible to obtain a similar uniqueness result as in  the previous theorem. Indeed, it is enough to consider  the well-known warped model of the De Sitter spacetime (see for instance \cite{AledoAlias2001}).
		\end{remark}

		\section{Calabi-Bernstein-type results}
		\label{Calabisection}

		A characterization of maximal compact
		hypersurfaces has been obtained in
		Theorem \ref{CharactMaximal1}  
		for expanding and contracting twisted product spacetimes $I\times_f F$, and also for those which admit a transition slice. 
		In this section we focus on  
		immersed spacelike hypersurfaces $\psi:\Sigma\rightarrow I\times_f F$ which are graphs over the fiber $(F,g_F)$.
		The calculation of the graph mean curvature $H$ becomes essential to reformulate the maximality results of  
		Theorem \ref{CharactMaximal1} in a proper Calabi-Bernstein style by providing existence and uniqueness results to the associated $H=0$ elliptic PDE under suitable conditions. To this purpose, let $\Omega\subset F$ be an open set in $I\times_f F$. Any $u\in C^{\infty}(\Omega)$ defines 
		a local graph $\Sigma_u=\{(u(p),p) : p\in \Omega\}$ on the fiber, where the map $x:\Omega\subset F \rightarrow I\times_f F$ given by $x(p)=(u(p),p)$ is an immersion. Let us refer to the induced metric on the graph as $g_u$, which is of the form (\ref{inducedmetricexpr}), where the time-height function $\tau$ of the graph is related to $u$ by $u=\tau\circ x$. The pull-back of $g_u$ by $x$ to the fiber $(F,g_F)$ reads
		\begin{equation}
			\label{inducedmetric}
			x^{*}(g_u)=-du^2+(f\circ x)^2 g_F .
		\end{equation}  
 	The hypersurface $(\Sigma_u,g_u)$ is said to be an entire graph when $\pi_{F}\vert_{\Sigma_u}=x^{-1}:\Sigma_u\rightarrow F$ is a global diffeomorphism.
	Lemma \ref{Coveringmap} shows that $\pi_{F}\vert_{\Sigma_u}$
	is a cover map provided $(\Sigma_u,g_u)$ is complete and the restriction of $f$ to $(\Sigma_u,g_u)$ is bounded. 
	In particular, this occurs when  
	$I\times_f F$ admits a compact spacelike hypersurface $\psi:\Sigma\rightarrow I\times_f F$,
	which also implies the compactness of $(F,g_F)$. 
	Additionally, if the fiber is connected, $\pi_{F}\vert_{\Sigma}$ is a global diffeomorphism and hence $(\Sigma_u,g_u)$ is a global graph over the fiber. In the following proposition the causal character of spacelike graphs is characterized in terms of the gradient of the graph function $u$. 
		\begin{Prop}
			Let $\Omega\subset F$ be an open set and $(\Sigma_u,g_u)$ a graph over $(\Omega,g_F)$ determined by the immersion 
			$x:\Omega\subset F\rightarrow I\times_f F$, where $x(p)=(u(p),p)\in I\times_f F$ for some 
			function $u\in C^\infty(\Omega)$. Then the validity of the inequality 
			\begin{equation}
				\label{graphspacelikefiber}
				\vert\nabla^F u\vert < f\circ x
			\end{equation}
			at every $p\in \Omega$ is necessary and sufficient for the graph metric $g_u$ to be Riemannian, where $\nabla^F$ is the Levi-Civita connection of $(F,g_F)$.
		\end{Prop}
		\begin{proof}
			For every $p\in \Omega$, let us consider the push-forward of the gradient $\nabla^F u\in \mathfrak{X}(F)$ 
			by $dx$ to the tangent bundle of $(\Sigma_u,g_u)$ 
			 \begin{equation}
			 	\label{pushedforwardrelation}
			 	dx(\nabla^F u)=(\vert \nabla^F u \vert^2_{g_F}\circ x^{-1})\partial_t+\nabla^F u.
			 \end{equation}
			Squaring the above relation gives
			 \begin{equation}
			 	\label{squaringdiffgradu}
			 	\vert dx(\nabla^F u) \vert^2=(x^{-1})^{*}\left( \vert \nabla^F u \vert^2_{g_F}((f\circ x)^2-\vert \nabla^F u \vert^2_{g_F}) \right),
			 \end{equation}
          where $(x^{-1})^{*}$ stands for the pull-back by $x^{-1}:(\Sigma_u,g_u)\rightarrow (F,g_F)$.   		 
		 Since $g_u$ is a Riemannian metric, the right-hand side of (\ref{squaringdiffgradu}) becomes strictly positive when $\nabla^F u\neq 0$, which can only happen if (\ref{graphspacelikefiber}) holds. On the other hand, 
		  (\ref{graphspacelikefiber}) is also valid when $\nabla^F u=0$, so in any case (\ref{graphspacelikefiber}) holds true.     	 
		\end{proof}

		The mean curvature of the graph $(\Sigma_u,g_u)$ satisfies a PDE which involves the function $u: \Omega\subset F\rightarrow \mathbb{R}$ and its first and second order derivatives. This expression is well known when the ambient is a warped product (see  \cite{AledoRubioSalamanca2017} for instance). In this section
		the corresponding maximal equation for a graph in $I\times_f F$ is derived from (\ref{laptau}). 
		To this purpose, we first obtain 
		in the following lemma  the expression 
		for the Laplacian of $\tau$  in terms of the fiber geometry:
		\begin{Lem}
			\label{laplacianRelationLemma}
			Let $\Omega\subset F$ be an open set and $(\Sigma_u,g_u)$ a graph over $(\Omega,g_F)$ determined by the immersion 
			$x:\Omega\subset F\rightarrow I\times_f F$, where $x(p)=(u(p),p)\in I\times_f F$ for some 
			function $u\in C^\infty(\Omega)$. Assume that 
			$\vert\nabla^{F}u\vert < f\circ x$. The expression for $\triangle \tau$ in terms of the fiber geometry is
			\begin{equation}
				\label{laplacianRelation}
				x^*(\triangle \tau)=\fun(f\circ x)^{-n+2} g_F(\nabla^{F}(\fun(f\circ x)^n),\nabla^{F}u)+\fun^2(f\circ x)^{2} \triangle_{F} u \, ,
			\end{equation} 
			where  $\rho:\Omega\subset F\rightarrow \mathbb{R}$ is defined as $\fun=\frac{1}{(f\circ x)\sqrt{(f\circ x)^2-\vert\nabla^{F}u\vert^2}}$ and $\triangle_{F}$ is the Laplacian on $(F,g_F)$.
		\end{Lem}

		\begin{proof}
			Let $\{E_i\}_{i=1}^{n}$ be a local basis tangent to the
			graph $(\Sigma_u,g_u)$.
			The decomposition of $E_i$ into its horizontal and vertical parts
			is $E_i=(\partial_i u)\partial_t+E_i^F$.
			The Laplacian of $\tau$ can be computed using formula (\ref{lapconforme}) in the Appendix, which in this context reads
			\begin{equation}
				\label{lapgraph1}
				\triangle\tau=\frac{1}{\sqrt{\det g_u}}\sum_{i=1}^n E_i\left(\sqrt{\det g_u} \, (\nabla \tau)^i \right),
			\end{equation} 
			where $(\nabla \tau)^i$ stands for the i-component of $\nabla\tau$ in the basis $\{E_i\}_{i=1}^n$.
			A straightforward computation shows that the components of $\nabla \tau$ and $\nabla^{F}u$ in the respective bases $\{E_i\}_{i=1}^n$ and $\{E_i^F\}_{i=1}^n$
			are related by
			\begin{equation}
				\label{gradrelation}
				x^{*}((\nabla \tau)^i)= f^2\fun^2 (\nabla^{F}u)^i, 
			\end{equation}
		    where $\rho$ is defined as in the statement of this lemma. On the other hand, the metrics $g_u$ and $g_F$ are related by (\ref{inducedmetric}), which makes their determinants satisfy   
			\begin{equation}
				\label{detrelation}
				x^{*}(\det g_u)=\fun^{-2}(f\circ x)^{2n-4} \det g_F. 
			\end{equation}
			Plugging (\ref{gradrelation}) and (\ref{detrelation}) into (\ref{lapgraph1}) gives
			\begin{equation*}
				x^*(\triangle\tau)=\frac{1}{(f\circ x)^{n-2} \fun^{-1}\sqrt{\det g_F}}\sum_{i=1}^n E_i^F\left(\fun(f\circ x)^{n} \sqrt{\det g_F}\, (\nabla^{F} u)^i \right),
			\end{equation*}
			where we have used that  $X(h)=X^F(h\circ x)$ for any function $h:\Sigma_u\rightarrow\mathbb{R}$ and any $X=X^F(u)\partial_t+X^F$ tangent to $(\Sigma_u,g_u)$. 
			Hence
			\begin{eqnarray*}
				x^*(\triangle\tau)&=& (f\circ x)^{-n+2}\fun \left(
				\left(\sum_{i=1}^n E_i^F\left(\fun(f\circ x)^{n}\right)(\nabla^{F}u)^i\right)+\frac{\fun(f\circ x)^n}{\sqrt{\det g_F}} \sum_{i=1}^n E_i^F\left(\sqrt{\det g_F} (\nabla^{F}u)^i\right) \right) \nonumber \\
				&=& \fun(f\circ x)^{-n+2} g_F(\nabla^{F}(\fun(f\circ x)^n),\nabla^{F}u)+\fun^2(f\circ x)^{2} \triangle_{F} u,  \label{lapgraph2}  
			\end{eqnarray*}
			as stated above.
		\end{proof}

		The expression for the mean curvature $H$ of a graph over the fiber in $I\times_f F$ follows directly from  
		Lemma \ref{laplacianRelationLemma}:
		\begin{Prop}
			\label{Propmeancurvequation}
			Let $\Omega\subset F$ be an open set and $(\Sigma_u,g_u)$ a graph over $(\Omega,g_F)$ determined by the immersion 
			$x:\Omega\subset F\rightarrow I\times_f F$, where $x(p)=(u(p),p)\in I\times_f F$ for some 
			function $u\in C^\infty(\Omega)$. Assume that 
			$\vert\nabla^{F}u\vert < f\circ x$. Its mean curvature $H$ satisfies the following equation:
			\begin{eqnarray}
				\label{meancurvaturefibergeom}
				n\cdot x^*H=\div_{F}(\fun \nabla^{F}u)+(f\circ x)^2\fun\left(n+\frac{\vert\nabla^{F}u\vert^2_{g_F}}{(f\circ x)^2}\right)(\der\log f)\circ x+n\fun g_F(\nabla^{F}\log( f\circ x),\nabla^{F}u),   
			\end{eqnarray}	
			where $\rho:\Omega\subset F\rightarrow \mathbb{R}$ is defined as $\fun=\frac{1}{(f\circ x)\sqrt{(f\circ x)^2-\vert\nabla^{F}u\vert^2}}$ and $\div_{F}$ is the divergence operator with respect to $g_F$.
		\end{Prop}
		\begin{proof}
			Relation ($\ref{meancurvaturefibergeom}$) for $H$ can be derived by rewritting
			the whole equality (\ref{laptau}) in terms of the geometry of $(F,g_F)$ and solving for $H$. 
			The expression for $\triangle\tau$ in terms of the fiber geometry has already 
			been obtained in Lemma \ref{laplacianRelationLemma}, which together with 
		    relation $x^*(\vert \nabla\tau\vert^2 )=(f\circ x)^2 \rho^2\vert \nabla^F u\vert^2_{g_F}$ derived from  (\ref{gradrelation}) can be plugged into (\ref{laptau}) to obtain
			\begin{equation}
				\label{meanCurvEq1}
				(f\circ x)^{-n} g_F(\nabla^{F}(\fun(f\circ x)^n),\nabla^{F}u)+\fun\triangle_{F} u=\frac{-1}{\fun(f\circ x)^2}\left(n+\fun^2(f\circ x)^2\vert\nabla^{F}u\vert^2_{g_F}\right)(\der \log f)\circ x+n\cdot x^*H.
			\end{equation}
		    Notice that 
			\begin{equation*}
				\nabla^{F}(\fun(f\circ x)^n)=(\nabla^{F}\fun)(f\circ x)^n+n\fun(f\circ x)^{n-1}(\conb f\circ dx),
			\end{equation*}
			so the first term of the left-hand side of (\ref{meanCurvEq1}) becomes
			\begin{equation*}
				(f\circ x)^{-n}g_F(\nabla^{F}(\fun(f\circ x)^n),\nabla^{F}u)=g_F(\nabla^{F}\fun,\nabla^{F}u)
				+n\fun(f\circ x)^{-1}(\conb f\circ dx)(\nabla^{F}u).  
			\end{equation*}
			Using (\ref{pushedforwardrelation}) in combination with formula $\textrm{div}_{F}(\fun \nabla^{F}u)=g_F(\nabla^{F}\fun,\nabla^{F}u)+\fun\triangle_{F} u$ finally gives (\ref{meancurvaturefibergeom})
			after solving for $H$ in (\ref{meanCurvEq1}),
			\begin{eqnarray*}
				\label{meanCurvEq2}
				n\cdot x^*H&=& \nonumber \\
				&=&\mathrm{div}_{F}(\fun \nabla^{F}u)+\left(\left(\frac{1}{\fun(f\circ x)^2}+\fun\vert\nabla^{F}u\vert^2_{g_F}\right)n+\fun\vert\nabla^{F}u\vert^2_{g_F}\right)(\partial_t\log f)\circ x  \nonumber \\
				&&+n\fun g_F(\nabla^{F}(\log f\circ x),\nabla^{F}u) \nonumber \\
				&=&\div_{F}(\fun \nabla^{F}u)+(f\circ x)^2\fun\left(n+\frac{\vert\nabla^{F}u\vert^2_{g_F}}{(f\circ x)^2}\right)(\der\log f)\circ x+n\fun g_F(\nabla^{F}\log( f\circ x),\nabla^{F}u).
			\end{eqnarray*}
		\end{proof}
		
		\begin{remark}
			It is immediate to check that relation (\ref{meancurvaturefibergeom}) for the mean curvature of graphs over the fiber in a twisted product context reduces to the one given in \cite{AledoRubioSalamanca2017} for GRW spacetimes when the twisted function is assumed to depend just on the time variable.   
		\end{remark}

		The results obtained in Theorem \ref{CharactMaximal1} for hypersurfaces in $I\times_f F$ whose mean curvature satisfies  
		$H=0$ with $H$ as in (\ref{meancurvaturefibergeom})
		are Calabi-Bernstein-type, since the only maximal ones 
	    under the stated expanding conditions    
		are the slices. The precise formulation for these results in a proper Calabi-Bernstein style reads as follows:
		\begin{Tma} Let $(M,g)$ be a connected compact Riemannian $n$-manifold where $D$ is the Levi-Civita connection of $g$. Let $I\subset\mathbb{R}$ be an open interval and $f:I\times M\longrightarrow \mathbb{R}$ a positive smooth function satisfying some of the following possibilities:
			
			i) $\partial_t f\leq 0$, for all $t\in I$, or
			
			ii)  $\partial_t f\geq 0$, for all $t\in I$, or
			
			iii) there exists $t_1\in I$ such that,  $\partial_t f\geq 0$, for all $t\leq t_1$ and $\partial_t f\leq 0$, for all 
			$t\geq t_1$.

			Then, the only solution to the elliptic non-linear problem 
			\begin{eqnarray}
				\label{maximalgraphequation}
				\div (\fun Du)+(f\circ x)^2\fun\left(n+\frac{|Du|^2}{(f\circ x)^2}\right)(\der\log f)\circ x+n\fun g_F(D\log( f\circ x),Du) =0,  
				\nonumber 
			\end{eqnarray}
			where $x:M\rightarrow I\times M$, $x(p)=(u(p),p)$, $\rho=\frac{1}{(f\circ x)\sqrt{(f\circ x)^2-|Du|^2}}$ and $|Du|<f\circ x$
			are the constant functions $u=t_0$, with $\partial_t f(t_0,p)=0$, for all $p\in M$.

		\end{Tma}

\section{Appendix}

We devote this Appendix to compile elementary results within the area of conformal geometry
which are useful for this work. Specifically, given a non-degenerate submanifold
in a pseudo-Riemannian background, we show the relation between the
connections, the Weingarten operators, the mean curvatures and the Laplacians associated to two conformally
related metrics. Let us first introduce the concept of conformal transformation.
\begin{Def}
	Let $(M,g_1)$ be a pseudo-Riemannian manifold of dimension $n=\mathrm{dim}M \geq 2$. A conformal transformation
	is defined as a transformation of the metric tensor of the form
	\begin{equation*}
		g_1 \rightarrow g_2=e^{2\phi}g_1, 
	\end{equation*}
where $\phi\in C^\infty(M)$.
\end{Def}
In the particular case where $g_1$ has Lorentzian signature, these transformations preserve the angles between
geometrical objects, as well as null geodesics and light cones. The following result establishes the preservation of  
non-degenerate totally umbilical submanifolds in pseudo-Riemannian contexts  
under the action of a global conformal transformation.    
\begin{Prop}
	\label{Propconformaltransf}
	Let $\psi:\Sigma\rightarrow (M,g_1)$ be a non-degenerate totally umbilic submanifold of arbitrary codimension immersed in a pseudo-Riemannian manifold $(M,g_1)$. Let us consider the conformal transformation defined by 
	\begin{equation}
		\label{conformaltransf}
		g_2=e^{2\phi} \,g_1, 
	\end{equation}
	with $\phi\in C^\infty(M)$.
	Then $\Hyp$ is totally umbilical in $M$ when is endowed with the metric $g_2$.   
\end{Prop}        

\begin{proof}
	Let us denote by $\nabla^{(2)}$ the Levi-Civita connection of the metric $g_2$.
	For any $X,Y\in \mathfrak{X}(\Sigma)$,
	$\nabla^{(2)}$ is related to the ambient connection $\nabla^{(1)}$ by the following formula (see for instance \cite{Dajczer}, pg. 132):
	\begin{equation}
		\label{conrelation}
		\nabla^{(2)}_X Y=\nabla^{(1)}_X Y+g_1( X, \nabla^{(1)} \phi ) Y+g_1( Y, \nabla^{(1)} \phi ) X- g_1( X,Y ) \nabla^{(1)} \phi.
	\end{equation}
    Given any unit normal $N_1$ to $\Hyp$ in $(M,g_1)$, relation (\ref{conrelation}) can be applied to compute the covariant derivative of the associated unit normal $N_2=e^{-\phi} N_1$ in $(M,g_2)$ in terms of the geometry determined by $g_1$, 
	\begin{equation*}
		\nabla^{(2)}_X N_2= e^{-\phi} \left( \nabla^{(1)}_X N_1+g_1(N_1, \nabla^{(1)} \phi)X \right).
	\end{equation*} 
    Hence, the relation between the shape operators $A_1$ and $A_2$ associated to $N_1$ and $N_2$ becomes
	\begin{equation}
		\label{shaperelation2}
		A_2(X)=e^{-\phi}\left(A_1(X)-g_1(N_1,\nabla^{(1)} \phi)\right). 
	\end{equation}
    The shape operator of the totally umbilic non-degenerate submanifold $\psi: \Sigma\rightarrow (M,g_1)$ reads $A_1(X)=-H_1\textrm{Id}$, with $H_1$ the mean curvature of $\psi:\Hyp\rightarrow (M,g_1)$ with respect to $N_1$,
	so formula (\ref{shaperelation2}) becomes in this case
	\begin{equation*}
		A_2(X)=e^{-\phi}\left(-H_1-g_1(N_1,\nabla^{(1)} \phi)\right)\textrm{Id}(X).  
	\end{equation*}
\end{proof}

\begin{Cor}
	\label{charactmaximal}
	Under the conformal transformation (\ref{conformaltransf}), the mean curvature function $H_2$ of $\psi:\Sigma\rightarrow (M,g_2=e^{2\phi}g_1)$ with respect to $N_2$ becomes 
	\begin{equation}
		\label{confmeancurvature}
		H_2=e^{-\phi}\left(H_1+g_1(N_1,\nabla^{(1)} \phi)\right). 
	\end{equation}
	Moreover, if the mean curvature $H_1$ of the hypersurface $\psi:\Sigma\rightarrow (M,g_1)$ satisfies the equation
	\begin{equation*}
		H_1=-g_1(N_1,\nabla^{(1)} \phi),
	\end{equation*}
	then $H_2=0$. 
\end{Cor}
		
 The relation between conformal Laplacians can be easily found in the literature
(see for example \cite{Besse}). For the sake of completeness of this text, we add in this Appendix a proof for this result.   
\begin{Prop}
Let $(M,g_1)$ be a pseudo-Riemannian manifold of dimension $n=\mathrm{dim}\,M \geq 2$.
Under the conformal transformation determined by 
$g_2=e^{2\phi}g_1$ for some $\phi\in C^\infty(M)$, the associated Laplacians  $\triangle_{(1)}$ and $\triangle_{(2)}$ satisfy the relation
\begin{equation}
	\label{conformallaplacianrelation}
	\triangle_{(2)}h=e^{-2\phi}(\triangle_{(1)}h+(n-2)g_1(\nabla^{(1)} \phi, \nabla^{(1)} h) )
\end{equation}
for any $h\in C^\infty(M)$.
\end{Prop}
		
\begin{proof}
Let $\{x^i\}_{i=1}^{n}$ be a system of local coordinates on $M$.
The Laplacian of $h$ with respect to $g_2$ can be computed by the well-known formula (read for instance \cite{TorresdelCastillo}) 
	\begin{equation}
		\label{lapconforme}
		\triangle_{(2)}h=\frac{1}{\sqrt{\det g_2}}\sum_{i=1}^n \partial_i\left(\sqrt{\det g_2} \, (\nabla^{(2)} h)^i \right),
	\end{equation} 
where $(\nabla^{(2)} h)^i$ stands for the i-component of $\nabla^{(2)} h$ in the associated coordinate basis $\{\partial/\partial x^i \}_{i=1}^n$. The determinants of both metrics satisfy $\mathrm{det}g_2=e^{2\phi n}\mathrm{det}g_1$, which inserted into (\ref{lapconforme}) gives
\begin{eqnarray*}
	\triangle_{(2)}	h&=&\frac{1}{e^{\phi n}\sqrt{\mathrm{det}g_1}}\sum_{i=1}^n \partial_i\left(e^{\phi(n-2)}\sqrt{\mathrm{det}g_1}\,\, (\nabla^{(1)} h)^i \right) \\
	&=& \frac{1}{e^{\phi n}\sqrt{\mathrm{det}g_1}}\sum_{i=1}^n e^{\phi(n-2)}\left( (n-2)\partial_i \phi \sqrt{\mathrm{det}g_1}\,\,(\nabla^{(1)} h)^i +\partial_i (\sqrt{\mathrm{det}g_1}\,\,(\nabla^{(1)} h)^i)   \right) \\
	&=& e^{-2\phi}(\triangle_{(1)}h+(n-2)g_1(\nabla^{(1)} \phi, \nabla^{(1)} h) ).
\end{eqnarray*}

\end{proof}

\end{document}